\documentclass{amsart}
\usepackage{amsthm,amssymb,amsmath,thmtools}
\usepackage{enumitem}
\usepackage{hyperref}
\newtheorem{theorem}{Theorem}[section]
\newtheorem{lemma}[theorem]{Lemma}

\newtheorem{corollary}{Corollary}[section]

\theoremstyle{definition}

\theoremstyle{remark}

\numberwithin{equation}{section}

\hypersetup{
colorlinks=true,
linkcolor=black
}
\begin{document}
\title[fractional fast diffusion equations]{On the finite element approximation for fractional fast diffusion equations}
\author[D. Li]{Dongxue Li}
\address{\noindent School of Mathematics, Tianjin University, Tianjin 300072, P. R. China}
\email{lidongxue@tju.edu.cn}
\author[Y. Zheng]{Youquan Zheng}
\address{\noindent School of Mathematics, Tianjin University, Tianjin 300072, P. R. China}
\email{zhengyq@tju.edu.cn}

\begin{abstract}
Considering fractional fast diffusion equations on bounded open polyhedral domains in $\mathbb{R}^N$, we give a fully Galerkin approximation of the solutions by $C^0$-piecewise linear finite elements in space and backward Euler discretization in time, a priori estimates and the rates of convergence for the approximate solutions are proved, which extends the results of \emph{Carsten Ebmeyer and Wen Bin Liu, SIAM J. Numer. Anal., 46(2008), pp. 2393--2410}. We also generalize the a priori estimates and the rates of convergence to a parabolic integral equation under the framework of \emph{Qiang Du, Max Gunzburger, Richaed B. Lehoucq and Kun Zhou,
SIAM Rev., 54 (2012), no. 4, pp. 667--696.}
\end{abstract}
\maketitle
\section{Introduction}
Suppose $\Omega\subset\mathbb{R}^N$ is a bounded open polyhedral domain and $T \in (0, +\infty)$, we consider the finite element approximation of the following nonlocal nonlinear diffusion problem,
\begin{equation}\label{E:main}
\begin{cases}
u_t = - (-\Delta)^s(|u|^{m-1}u) &\quad \text{ in }\,\, \Omega\times (0,T],\\
u(x, t) = 0  &\quad \text{ on }\,\,  (\mathbb{R}^N \setminus \Omega)\times (0,T],\\
u(x, 0) = u_0(x) &\quad \text{ in }\,\,  \Omega,
\end{cases}
\end{equation}
for a function $u(x, t):\mathbb{R}^N \times[0,T]\to \mathbb{R}$, here $u_0(x) = 0$ on $\mathbb{R}^N\setminus\Omega$ and $s\in (0, 1)$.

The fractional Laplacian appeared in (\ref{E:main}) is defined as follows
$$(-\Delta)^sg(x) = c_{N,s}\text{p.v.}\int_{\mathbb{R}^N}\frac{g(x) - g(z)}{|x - z|^{N+2s}}\, dz,$$
where $c_{N,s} > 0$ is a normalization constant, which can be viewed as an infinitesimal generator of the stable and radially symmetric L\'{e}vy processes \cite{Bertoin1996}. $(-\Delta)^s$ and its generalizations not only play important roles in probability, but also have been widely used to model nonlocal Dirichlet forms \cite{Applebaum2004}, phase transitions \cite{BatesChmaj1999}, nonlocal heat conduction \cite{Bobaru2010}, anomalous diffusions in physics \cite{MetzlerKlafter2000}, and among others in recent years. We refer the interested readers to the well written paper \cite{DuqiangSIAMReview} for further discussions on more applications as well as related mathematical work.

When $s = 1$, $u_t=\Delta(|u|^{m-1}u)$ is the classical fast diffusion equation for $0<m<1$ and the porous medium equation for $1<m<\infty$, which are degenerate or singular equations, many standard finite element techniques do not work well. The a priori error bounds are often suboptimal and the orders usually converge to zero when parameters approximate zero or infinity. The first error estimates of fully discrete schemes for porous medium equations were obtained in \cite{Rose2010}. Semi-discretizations in time for the porous medium and fast diffusion equations were discussed in \cite{EdenMichauxRakotoson}, \cite{Roux1989} and \cite{RouxMainge}. Fully discrete schemes for the fast diffusion equations were given in \cite{LeftonWei} and \cite{RullaWalkington}. In \cite{Ebmeyer-Liu-2008}, the authors used the $C^0$-piecewise linear finite elements in space and the backward Euler time discretization, a priori error estimates in quasi-norms were derived, furthermore, they estimated the rate of convergence. In \cite{NogHuang2017}, an adaptive moving mesh finite element method was studied for the porous medium equations.

In this paper, we consider the fractional case $s\in (0, 1)$, which also includes the fast diffusion equations $(0 < m < 1)$ and the porous medium equations $(1 < m < \infty)$, respectively. The case $m = 1$ has attracted much attention in recent years, we limit ourselves to cite only, for example, \cite{DuqiangSIAMReview}, \cite{DuGuZhou2013}, \cite{DuHuangLehoucq2014}, \cite{DuYang2016}, \cite{DuYangZhou2017}, \cite{DuZhou2011}. Reference \cite{DuZhou2011} established a nonlocal functional analytical framework for a linear peridynamic model of a spring network system, Galerkin finite element approximation was applied for numerical approximation. In \cite{DuGuZhou2013}, the nonlocal analogs of first and second Green's identities, local and nonlocal balance laws and second-order elliptic boundary-value problems, to the classical calculus were established. By exploiting the nonlocal vector calculus given in \cite{DuGuZhou2013}, the authors in \cite{DuqiangSIAMReview} provided a variational analysis for nonlocal diffusion problems, which were described by a class of parabolic linear integral equation.
Finite dimensional approximations using continuous and discontinuous Galerkin methods, condition number and error estimates were also derived. In \cite{DuHuangLehoucq2014}, by introducing a nonlocal operator with the indicator function of a domain, the authors gave a description of initial and initial volume-constrained problems associated with a linear nonlocal convection-diffusion equation. Finite difference schemes as well as Monte Carlo simulations were applied to solve such nonlocal problems. In \cite{DuYang2016}, the Fourier spectral approximations of a nonlocal Allen-Cahn equation were investigated, the authors also provided various error estimates and studied the steady states of the models with different kernels via both numerical simulations and theoretical analysis. References \cite{DuYangZhou2017} studied a nonlocal-in-time parabolic equation, a semi-discrete finite element approximation was proposed and error estimates were obtained.

For the case $s\in (0, 1)$ and $m\neq 1$, in the well written papers \cite{deltesoJakobsen2019} and \cite{deltesoJakobsen2018}, the authors dealt with fully discrete numerical methods and provided in \cite{deltesoJakobsen2019} rigorous analysis of such numerical schemes which covers local and nonlocal, linear and nonlinear, non-degenerate and degenerate, and smooth and non-smooth problems. In \cite{deltesoJakobsen2018}, the authors gave concrete discretizations and verify numerically some important theoretical properties of the methods given in \cite{deltesoJakobsen2019}. Regularity of solutions for (\ref{E:main}) were studied, for example, in \cite{BonFigalliRoOton} and \cite{VPFR2017}.

In this paper we use the methods developed in \cite{Ebmeyer-Liu-2008} to study problem (\ref{E:main}) for $m\in (0, 1)$, in Section 2, we give a Galerkin approximation of the solutions for (\ref{E:main}) by $C^0$-piecewise linear finite elements in space and backward Euler discretization in time, in Section 3, we prove a priori estimates for the approximate solutions and in Section 4, the rates of convergence is showed, which extends the results of \cite{Ebmeyer-Liu-2008} to the fractional case (\ref{E:main}). In Section 5, we show that the results in Section 3 and Section 4 can be generalized to a more general problem then (\ref{E:main}) described by parabolic integral equations in the framework of \cite{DuqiangSIAMReview}.
\section{Discretization of problem (\ref{E:main})}\label{Section2}
To get a finite element approximation of problem (\ref{E:main}), we use the following assumptions:
\begin{enumerate}
\item[(H1)] $\Omega \subset \mathbb{R}^N$, $N \ge 1$, is a bounded, convex polyhedral domain.

\item[(H2)] $\partial\Omega = \bigcup_{1\leq k \leq M} \overline{\Gamma}_k$, each $\Gamma_k$ is a $(N - 1)$-dimensional polyhedron and $\Gamma_i\cap\Gamma_k=\phi$ for $i\neq k$.

\item[(H3)] $\partial\Gamma_{k1}\cap \dots \cap \partial\Gamma_{kj} = \phi$ if $j > N$ and $k_1 < \dots <k_j$.

\item[(H4)] $u_0 \in L^{\infty}(\Omega)$ and $u_0 = 0$ on $\mathbb{R}^N\setminus\Omega$.
\end{enumerate}
$u(x,t)$ is called a weak solution of \eqref{E:main} if
\begin{equation}
\label{Eq3.1}
\begin{aligned}
& - \int_0^T\int_{\mathbb{R}^N}u\varphi_t\, dx\, dt \\
&\qquad + \int_0^T\int_{\mathbb{R}^N}\int_{\mathbb{R}^N}\frac{(|u|^{m-1}u)(x) - (|u|^{m-1}u)(z)}{|x - z|^{N+2s}}(\varphi(x) - \varphi(z))\, dz\, dx\, dt\\
& = \int_{\mathbb{R}^N} u_0\varphi_0\, dx
\end{aligned}
\end{equation}
holds, where $\varphi(\cdot, T) \equiv 0$ on $\Omega$, $\varphi_0 = \varphi(\cdot, 0)$.
Indeed, multiplying \eqref{E:main} by $\varphi$ and integrating it over $\mathbb{R}^N$, then
\[\int_0^T\int_{\mathbb{R}^N}u_t\varphi\, dx\, dt = - \int_0^T\int_{\mathbb{R}^N}(-\Delta)^s(|u|^{m-1}u)\varphi\, dx\, dt.\]
First, integrate by parts, we have
\[\int_0^T\int_{\mathbb{R}^N}u_t\varphi\, dx\, dt = - \int_{\mathbb{R}^N} u_0\varphi_0\, dx - \int_0^T \int_{\mathbb{R}^N}u\varphi_t\, dx\, dt.\]
Second, there holds
$$
\begin{aligned}
I &:= \int_0^T\int_{\mathbb{R}^N}(-\Delta)^s(|u|^{m-1}u)\varphi(x)\, dx\, dt\\
  &= 2\int_0^T\int_{\mathbb{R}^N}\int_{\mathbb{R}^N}\frac{(|u|^{m-1}u)(x) -(|u|^{m-1}u)(z)}{|x - z|^{N+2s}}\, dz\varphi(x)\ dx\, dt\\
  &= 2\int_0^T\int_{\mathbb{R}^N}\int_{\mathbb{R}^N}\frac{(|u|^{m-1}u)(x) - (|u|^{m-1}u)(z)}{|x - z|^{N+2s}}\varphi(x)\, dz\, dx\, dt\\
  &= 2\int_0^T\int_{\mathbb{R}^N}\int_{\mathbb{R}^N}\frac{(|u|^{m-1}u)(x) - (|u|^{m-1}u)(z)}{|x - z|^{N+2s}}\big(\varphi(x) - \varphi(z)\big)\, dz\, dx\, dt\\
  &\quad + 2\int_0^T\int_{\mathbb{R}^N}\int_{\mathbb{R}^N}\frac{(|u|^{m-1}u)(x) - (|u|^{m-1}u)(z)}{|x - z|^{N+2s}}\, dz\varphi(z)\, dz\, dx\, dt.
\end{aligned}
$$
Since
$$
\begin{aligned}
& 2\int_0^T\int_{\mathbb{R}^N}\int_{\mathbb{R}^N}\frac{(|u|^{m-1}u)(x) - (|u|^{m-1}u)(z)}{|x - z|^{N+2s}}\, dz\varphi(z)\, dz\, dx\, dt\\
& = -2\int_0^T\int_{\mathbb{R}^N}\int_{\mathbb{R}^N}\frac{(|u|^{m-1}u)(z) - (|u|^{m-1}u)(x)}{|x - z|^{N+2s}}\varphi(z)\, dx\, dz\, dt\\
& = -I,
\end{aligned}
$$
we get
$$I = \int_0^T\int_{\mathbb{R}^N}\int_{\mathbb{R}^N}\frac{(|u|^{m-1}u)(x) - (|u|^{m-1}u)(z)}{|x - z|^{N+2s}}\big(\varphi(x) - \varphi(z)\big)\, dz\, dx\, dt.$$

\subsection{The discrete problem.}\label{Subsection2.2}
In the following, denote the time step as $\tau > 0$. For $n = 0, \cdots, M$, set $t_n = n\tau$, $T = t_M$, $I_n = (t_{n-1},t_n)$, $v_n(x) = v(x,t_n)$, $\bar{v}^0(x) = v_0(x)$ and
$$\bar{v}^n(x) = \tau^{-1}\int_{I_n}v(x,s)\, ds\text{ for } n \ge 1.$$
Let $T_h$ be a family of decompositions of $\Omega$ into closed $N-$simplices and $h$ is the mesh-size, moreover, assume $T_h$ is a regular triangulation, i.e., intersection of two non-disjoint nonidentical elements in $T_h$ is a common vertex or edge or surface, and there exists a constant $c > 0$ such that
$$|K| \ge c(\text{diam} \,\,\, K)^N \text{ for all simplices }K \in T_h.$$
Define
\begin{center}
$S_h(\Omega) =$ \{$\phi_h \in C^0(\overline{\Omega}):\phi_h$ is piecewise linear w.r.t. $T_h,\phi_h = 0$ on $\mathbb{R}^N \setminus \Omega$\}
\end{center}
and $\Pi_hv \in S_h$ denote the $C^0$-piecewise linear interpolation of $v$.
For $s \in (0,1)$, the standard fractional-order Sobolev space is defined as
$$H^s(\mathbb{R}^N) := \{u \in L^2(\mathbb{R}^N) : \Vert u\Vert_{L^2(\mathbb{R}^N)} + |u|_{H^s(\mathbb{R}^N)} < \infty\},$$
where $$|u|_{H^s(\mathbb{R}^N)}^2 := \int_{\mathbb{R}^N}\int_{\mathbb{R}^N}\frac{(u(y) - u(x))^2}{|y - x|^{N+2s}}\, dy\, dx.$$
Let $P_h : H^s(\mathbb{R}^N) \to S_h$ be the $H^s$-projection onto $S_h$ defined as
$$ \int_{\mathbb{R}^N}\int_{\mathbb{R}^N}\frac{(v - P_hv)(x) - (v - P_hv)(z)}{|x -z |^{N+2s}}\big(\chi(x) - \chi(z)\big)\, dz\, dx = 0 \quad \forall\chi \in S_h.$$
Let $w := |u|^{m-1}u$, then $\partial_tu = - (- \Delta)^sw$ and
\begin{equation}
\label{Eq4.1}
\partial_t\psi(w) = - (- \Delta)^s w,
\end{equation}
where $\psi(y) := |y|^{\frac{1-m}{m}}y$. Let $W_n \in S_h$, $n = 1,2,\ldots$ be the solutions of the following system
\begin{equation}
\label{Eq4.2}
\begin{aligned}
& \left(\frac{\psi(W_n) - \psi(W_{n-1})}{\tau},\chi_n \right)\\
& \quad + \int_{\mathbb{R}^N}\int_{\mathbb{R}^N}\frac{W_n(x) - W_n(z)}{|x - z|^{N+2s}}\big(\chi_n(x) - \chi_n(z)\big)\, dz\, dx = 0 \quad \forall\chi_n \in S_h,
\end{aligned}
\end{equation}
where $W_0 := \psi^{-1}(\Pi_h\psi(w_0))$. The finite element approximation $W(x,t)$ of $w(x,t)$ is defined as follows
\begin{equation}
\label{Eq4.3}
W(x,t) =
\begin{cases}
W_0(x) &\quad \mathrm{if} \,\, t = 0,\\
W_n(x,t)&\quad \mathrm{if} \,\, t \in (t_{n-1},t_n],\,\, 1 \leq n \leq M.
\end{cases}
\end{equation}
Then we have
\begin{lemma}
There exist unique functions $W_1,\ldots,W_M \in S_h$ solving \eqref{Eq4.2}.
\end{lemma}
The proof of is similar to that of \cite[Lamme 3.1]{Ebmeyer-Liu-2008}, so we omit it here .
\begin{lemma}
There exists a positive constant $c$ such that
$$
\begin{aligned}
& \displaystyle\sup_{1\leq n\leq M}\Vert W_n\Vert_{L^{\frac{m+1}{m}}(\mathbb{R}^N)}^{\frac{m+1}{m}} + \tau\displaystyle\sum_{n=1}^M\int_{\mathbb{R}^N}\int_{\mathbb{R}^N}\frac{(W_n(x) - W_n(z))^2}{|x - z|^{N+2s}}\, dz\, dx \leq c.
\end{aligned}
$$
\end{lemma}
\begin{proof}
Choose $\chi_n = W_n$ in \eqref{Eq4.2} and take summation, we obtain
$$
\begin{aligned}
&\displaystyle\sum_{n=1}^M(\psi(W_n),W_n) + \tau\displaystyle\sum_{n=1}^M\int_{\mathbb{R}^N}\int_{\mathbb{R}^N}\frac{(W_n(x) - W_n(z))^2}{|x - z|^{N+2s}}\, dy\, dx = \displaystyle\sum_{n=1}^M(\psi(W_{n-1}),W_n).
\end{aligned}
$$
From Young's inequality, $|\psi(W_{n-1})W_n| \leq \frac{1}{m+1}|W_{n-1}|^{\frac{m+1}{m}} + \frac{m}{m+1}| W_n|^{\frac{m+1}{m}}$. Since $\psi(W_n)W_n = |W_n|^{\frac{m+1}{m}}$,
$$
\begin{aligned}
&\frac{1}{m+1}\Vert W_M\Vert_{\frac{m+1}{m}}^{\frac{m+1}{m}} + \tau\displaystyle\sum_{n=1}^M\int_{\mathbb{R}^N}\int_{\mathbb{R}^N}\frac{(W_n(x) - W_n(z))^2}{|x - z|^{N+2s}}\, dy\, dx \leq \frac{1}{m+1}\Vert W_0\Vert_{\frac{m+1}{m}}^{\frac{m+1}{m}}.
\end{aligned}
$$
Clearly, $\Vert W_0\Vert_{\frac{m+1}{m}} \leq c$ for some positive constant $c$.
\end{proof}
\subsection{The quasi-norm.}\label{section4}
For $v_1,v_2 \in L^p(\Omega)$ and $p > 1$, define the quasi-norm as
$$\Vert v_2\Vert_{(v_1,p)}^2 := \int_{\Omega}(|v_1| + |v_2|)^{p-2}|v_2|^2.$$
From \cite{Ebmeyer-Liu-2008}, we know that
$$c_1\Vert v_2\Vert_{L^p(\Omega)}^p \leq \Vert v_2\Vert_{(v_1,p)} \leq \Vert v_2\Vert_{L^p(\Omega)}^{p/2}\qquad \text{for}\,\,\, 1 < p < 2$$
and
$$\Vert v_2\Vert_{L^p(\Omega)}^{p/2} \leq \Vert v_2\Vert_{(v_1,p)} \leq c_2\Vert v_2\Vert_{L^p(\Omega)}^p\qquad \text{for}\,\,\, 2 < p < \infty,$$
where $c_1,c_2 > 0$ are constants depending on $\Vert v_1\Vert_{L^p(\Omega)}$ and $\Vert v_2\Vert_{L^p(\Omega)}$. Also,
\begin{equation}
\label{Eq5.1}
\begin{aligned}
& \int_0^T\Vert w - v\Vert_{(w,\frac{m+1}{m})}^2 \cong \int_0^T(\psi(w) - \psi(v),w - v)\\
& \cong \int_0^T\Vert \psi(w) - \psi(v)\Vert_{(\psi(w),m+1)}^2 = \int_0^T\Vert u - \psi(v)\Vert_{(u,m+1)}^2
\end{aligned}
\end{equation}
for all $v \in L^{\frac{m+1}{m}}([0,T] \times \Omega)$, where $u = \psi(w)$.

Now consider the following time independent problem
$$\psi(v) +(-\Delta)^sv = f \quad \mathrm{on}\,\, \Omega, \qquad v = 0 \quad\mathrm{on}\,\, \mathbb{R}^N\setminus\Omega$$
for a smooth function $f$. It is easy to see that there exists a unique weak solution satisfying
\begin{equation}
\label{Eq5.2}
(\psi(v),\varphi) + \int_{\mathbb{R}^N}\int_{\mathbb{R}^N}\frac{v(x) - v(z)}{|x - z|^{N+2s}}(\varphi(x) - \varphi(z))\, dz\, dx = (f,\varphi)
\end{equation}
for $\forall\varphi \in H_0^s(\Omega) := \{u \in H^s(\mathbb{R}^N) : u \equiv 0 \quad \text{on} \,\,\, \mathbb{R}^N\setminus\Omega\}$.
Let $V \in S_h$ be the finite element approximation satisfying
\begin{equation}
\label{Eq5.3}
(\psi(V),\chi) + \int_{\mathbb{R}^N}\int_{\mathbb{R}^N}\frac{V(x) - V(z)}{|x - z|^{N+2s}}(\chi(x) - \chi(z))\, dz\, dx = (f,\chi)\text{ for }\forall\chi \in S_h.
\end{equation}
Then there exists a constant $c > 0$ independent of $h$ satisfying
\begin{equation}
\label{Eq5.4}
\begin{aligned}
& \Vert v - V\Vert_{(v,\frac{m+1}{m})}^2 + \int_{\mathbb{R}^N}\int_{\mathbb{R}^N}\frac{(v(x) - v(z) - (V(x) - V(z)))^2}{|x - z|^{N+2s}}\, dz\, dx\\
& \leq c\displaystyle\inf_{v_h \in S_h}\Big(\Vert v - v_h\Vert_{(v,\frac{m+1}{m})}^2 + \int_{\mathbb{R}^N}\int_{\mathbb{R}^N}\frac{(v(x) - v(z) - (v_h(x) - v_h(z)))^2}{|x - z|^{N+2s}}\, dz\, dx\Big).
\end{aligned}
\end{equation}
Indeed, choose $\varphi = \chi$ in \eqref{Eq5.2} and take difference between \eqref{Eq5.2} and \eqref{Eq5.3}, we obtain
\begin{equation}
\label{Eq5.5}
\begin{aligned}
& (\psi(v) - \psi(V),\chi) + \int_{\mathbb{R}^N}\int_{\mathbb{R}^N}\frac{v(x) - v(z) - (V(x) - V(z))}{|x - z|^{N+2s}}(\chi(x) - \chi(z))\, dz\, dx = 0.
\end{aligned}
\end{equation}
For $v_h \in S_h$, from \eqref{Eq5.5} we have
$$
\begin{aligned}
& (\psi(v) - \psi(V),v - V) + \int_{\mathbb{R}^N}\int_{\mathbb{R}^N}\frac{(v(x) - v(z) - (V(x) - V(z)))^2}{|x - z|^{N+2s}}\, dz\, dx\\
& = (\psi(v) - \psi(V),v - v_h) + \int_{\mathbb{R}^N}\int_{\mathbb{R}^N}\frac{v(x) - v(z) - (V(x) - V(z))}{|x - z|^{N+2s}}\\
& \qquad \times (v(x) - v(z) - (v_h(x) - v_h(z)))\, dz\, dx.
\end{aligned}
$$
From \cite[Lemma 4.4]{Ebmeyer-Liu-2008},
$$
\begin{aligned}
(\psi(v) - \psi(V),v - v_h) \leq \delta(\psi(v) - \psi(V),v - V) + c_{\delta}(\psi(v) - \psi(v_h),v - v_h).
\end{aligned}
$$
Then Young's inequality implies that
$$
\begin{aligned}
& \int_{\mathbb{R}^N}\int_{\mathbb{R}^N}\frac{v(x) - v(z) - (V(x) - V(z))}{|x - z|^{N+2s}}(v(x) - v(z) - (v_h(x) - v_h(z)))\, dz\, dx\\
& \leq \delta\int_{\mathbb{R}^N}\int_{\mathbb{R}^N}\frac{(v(x) - v(z) - (V(x) - V(z)))^2}{|x - z|^{N+2s}}\, dz\, dx \\
& \qquad + c_{\delta}\int_{\mathbb{R}^N}\int_{\mathbb{R}^N}\frac{(v(x) - v(z) - (v_h(x) - v_h(z)))^2}{|x - z|^{N+2s}}\, dz\, dx.
\end{aligned}
$$
Hence \eqref{Eq5.4} is valid.

\section{A priori error estimates.}\label{section5}
In this section we prove an a priori error estimate for $w - W$.
\begin{theorem}\label{thm:main5.1}
For any $m > 0$ there exists a constant $c > 0$ independent of $h$ and $\tau$ such that
$$
\begin{aligned}
& \int_0^T\Vert w - W\Vert_{(w,\frac{m+1}{m})}^2 + \int_{\mathbb{R}^N}\int_{\mathbb{R}^N}\frac{|(\bar{w}(x) - \bar{w}(z)) - (\overline{W}(x) - \overline{W}(z))|^2}{|x - z|^{N+2s}}\, dz\, dx\\
& \leq c\Bigg(\displaystyle\sum_{n=1}^M\int_{I_n}\Vert w_n - w\Vert_{(w,\frac{m+1}{m})}^2 + \int_0^T\Vert w - P_hw\Vert_{(w,\frac{m+1}{m})}^2 + \Vert \psi(w_0) - \Pi_h\psi(w_0)\Vert_2^2\\
&\quad + \int_{\mathbb{R}^N}\int_{\mathbb{R}^N}\frac{\left|\tau\displaystyle\sum_{n=1}^M\left[(\bar{w}^n(x) - \bar{w}^n(z)) - (P_h\bar{w}^n(x) - P_h\bar{w}^n(z))\right]\right|^2}{|x - z|^{N+2s}}\, dz\, dx  \Bigg),
\end{aligned}
$$
where $\bar{w}(x) = \int_0^Tw(x,t)\, dt$ and $\overline{W}(x) = \int_0^TW(x,t)\, dt$.
\end{theorem}
\begin{proof}
Integrating \eqref{Eq4.1} over $I_n$, we find
\begin{equation}
\label{Eq6.1}
\tau^{-1}(\psi(w_n) - \psi(w_{n-1})) + \int_{\mathbb{R}^N}\frac{\bar{w}^n(x) - \bar{w}^n(z)}{|x - z|^{N+2s}}\, dz = 0.
\end{equation}
Let $\chi_n \in S_h$, multiply \eqref{Eq6.1} and \eqref{Eq4.2} by $\tau\chi_n$, take difference and sum over $n$, we get
$$
\begin{aligned}
& \displaystyle\sum_{n=1}^M(\psi(w_n) - \psi(w_{n-1}) - (\psi(W_n) - \psi(W_{n-1})),\chi_n)\\
& \quad + \displaystyle\sum_{n=1}^M\tau\int_{\mathbb{R}^N}\int_{\mathbb{R}^N}\frac{\bar{w}^n(x) - \bar{w}^n(z) - (W_n(x) - W_n(z))}{|x - z|^{N+2s}}(\chi_n(x) - \chi_n(z))\, dz\, dx = 0,
\end{aligned}
$$
From the identity $\displaystyle\sum_{n=1}^M(a_n - a_{n-1})b_n = a_Mb_M - a_0b_0 + \displaystyle\sum_{n=0}^{M-1}a_n(b_n - b_{n+1})$, we obtain
$$
\begin{aligned}
& (\psi(w_M) - \psi(W_M),\chi_M) + \displaystyle\sum_{n=0}^{M-1}(\psi(w_n) - \psi(W_n),\chi_n - \chi_{n+1}) \\
& \quad +\displaystyle\sum_{n=1}^M\tau\int_{\mathbb{R}^N}\int_{\mathbb{R}^N}\frac{\bar{w}^n(x) - \bar{w}^n(z) - (W_n(x) - W_n(z))}{|x - z|^{N+2s}}(\chi_n(x) - \chi_n(z))\, dz\, dx \\
& = (\psi(w_0) - \psi(W_0),\chi_0).
\end{aligned}
$$
Set $\chi_n = \tau\displaystyle\sum_{k=n}^M(P_h\bar{w}^k - \Pi_hW_k)$ for $0 \leq n \leq M$, since $\Pi_hW_k = W_k$ for all $k \ge 1$ and $\chi_n - \chi_{n+1} = \tau(P_h\bar{w}^n - \Pi_hW_n)$, we have
\begin{equation}
\label{Eq6.2}
\begin{aligned}
& \tau\displaystyle\sum_{n=0}^M(\psi(w_n) - \psi(W_n),P_h\bar{w}^n - \Pi_hW_n) \\
& \quad + \tau\displaystyle\sum_{n=1}^M\int_{\mathbb{R}^N}\int_{\mathbb{R}^N}\frac{\bar{w}^n(x) - \bar{w}^n(z) - (W_n(x) - W_n(z))}{|x - z|^{N+2s}}\\
& \qquad \times \tau\displaystyle\sum_{k=n}^M(P_h\bar{w}^k(x) - P_h\bar{w}^k(z) - (W_k(x) - W_k(z)))\, dz\, dx \\
& = \left(\psi(w_0) - \psi(W_0), \tau\displaystyle\sum_{n=0}^M(P_h\bar{w}^n - \Pi_hW_n)\right).
\end{aligned}
\end{equation}
Subtracting $\tau(\psi(w_0) - \psi(W_0),P_h\bar{w}^0 -\Pi_hW_0)$ on both side of \eqref{Eq6.2}, from the identity
$$\displaystyle\sum_{n=1}^M\left(a_n\displaystyle\sum_{k=n}^Ma_k \right) = \frac{1}{2}\left(\displaystyle\sum_{n=1}^Ma_n \right)^2 + \frac{1}{2}\displaystyle\sum_{n=1}^M(a_n)^2$$
and the fact that $P_h$ is $H^s-$projection, we obtain
\begin{equation*}
\begin{aligned}
& \tau\displaystyle\sum_{n=1}^M\int_{\mathbb{R}^N}\int_{\mathbb{R}^N}\frac{\bar{w}^n(x) - \bar{w}^n(z) - (W_n(x) - W_n(z))}{|x - z|^{N+2s}}\\
& \quad \times \tau\displaystyle\sum_{k=n}^M(P_h\bar{w}^k(x) - P_h\bar{w}^k(z) - (W_k(x) - W_k(z)))\, dz\, dx \\
\end{aligned}
\end{equation*}
\begin{equation*}
\begin{aligned}
& = \tau\displaystyle\sum_{n=1}^M\int_{\mathbb{R}^N}\int_{\mathbb{R}^N}\frac{P_h\bar{w}^n(x) - P_h\bar{w}^n(z) - (W_n(x) - W_n(z))}{|x - z|^{N+2s}}\\
& \quad \times \tau\displaystyle\sum_{k=n}^M(P_h\bar{w}^k(x) - P_h\bar{w}^k(z) - (W_k(x) - W_k(z)))\, dz\, dx\\
& = \int_{\mathbb{R}^N}\int_{\mathbb{R}^N}\tau\displaystyle\sum_{n=1}^M\Bigg((P_h\bar{w}^n(x) - P_h\bar{w}^n(z) - (W_n(x) - W_n(z))\\
& \quad \times \tau\displaystyle\sum_{k=n}^M(P_h\bar{w}^k(x) - P_h\bar{w}^k(z) - (W_k(x) - W_k(z))\Bigg)\frac{1}{|x - z|^{N+2s}}\, dz\, dx\\
& \ge \int_{\mathbb{R}^N}\int_{\mathbb{R}^N}\frac{1}{2}\frac{\left|\tau\displaystyle\sum_{n=1}^M(P_h\bar{w}^n(x) - P_h\bar{w}^n(z) - (W_n(x) - W_n(z))\right|^2}{|x - z|^{N+2s}}\, dz\, dx.
\end{aligned}
\end{equation*}
Hence, from \eqref{Eq6.2} we get
\begin{equation}
\label{Eq6.3}
\begin{aligned}
J_1 + J_2 &:= \tau\displaystyle\sum_{n=1}^M(\psi(w_n) - \psi(W_n),w_n - W_n)\\
& \quad + \int_{\mathbb{R}^N}\int_{\mathbb{R}^N}\frac{1}{2}\frac{\left|\tau\displaystyle\sum_{n=1}^M(P_h\bar{w}^n(x) - P_h\bar{w}^n(z) - (W_n(x) - W_n(z)))\right|^2}{|x - z|^{N+2s}}\, dz\, dx\\
& \leq \tau\displaystyle\sum_{n=1}^M(\psi(w_n) - \psi(W_n),w_n - P_h\bar{w}^n)\\
& \quad + \left(\psi(w_0) - \psi(W_0), \tau\displaystyle\sum_{n=1}^M(P_h\bar{w}^n - W_n)\right)\\
& =: J_3 + J_4.
\end{aligned}
\end{equation}
By \cite[Lemma 4.4]{Ebmeyer-Liu-2008}, we have
\begin{equation*}
\begin{aligned}
J_3 &= \displaystyle\sum_{n=1}^M\int_{I_n}(\psi(w_n) - \psi(W_n),w_n - P_hw)\\
& \leq \delta\displaystyle\sum_{n=1}^M\int_{I_n}(\psi(w_n) - \psi(W_n),w_n - W_n)\\
&\quad + c_{\delta}\displaystyle\sum_{n=1}^M\int_{I_n}(\psi(w_n) - \psi(P_hw),w_n - P_hw)\\
& =: J_{31} + J_{32}.
\end{aligned}
\end{equation*}
For $\delta > 0$ sufficiently small, the term $J_{31}$ can be absorbed into the left-hand side of \eqref{Eq6.3}. From \cite[Lemma 4.3]{Ebmeyer-Liu-2008}, there holds
$$J_{32} \leq c\left[\displaystyle\sum_{n=1}^M\int_{I_n}(\psi(w_n) - \psi(w),w_n - w) + \displaystyle\sum_{n=1}^M\int_{I_n}(\psi(w) - \psi(P_hw),w - P_hw)\right].$$
H\"{o}lder's and Young's inequalities yield that
$$J_4 \leq c_{\delta}\Vert\psi(w_0) - \psi(W_0)\Vert_{L^2(\mathbb{R}^N)}^2 + \delta\left\Vert\tau\displaystyle\sum_{n=1}^M(P_h\bar{w}^n - W_n)\right\Vert_{L^2(\mathbb{R}^N)}^2.
$$
By \cite[Lemma 4.2, Lemma 4.3]{DuqiangSIAMReview}, the term $\delta\Vert\tau\displaystyle\sum_{n=1}^M(P_h\bar{w}^n - W_n)\Vert_2^2$ can be absorbed. Moreover, since $\psi(W_0) = \Pi_h\psi(w_0)$, $\Vert\psi(w_0) - \psi(W_0)\Vert_2^2 = \Vert\psi(w_0) - \Pi_h\psi(w_0)\Vert_2^2$.

Next, from \cite[Lemma 4.3]{Ebmeyer-Liu-2008}, we obtain
$$(\psi(w) - \psi(w_n),w - W_n) \leq c[(\psi(w) - \psi(w_n),w - w_n) + (\psi(w_n) - \psi(W_n),w_n - W_n)],$$
hence,
$$J_1 \ge c\left(\displaystyle\sum_{n=1}^M\int_{I_n}(\psi(w) - \psi(W_n),w - W_n) - \displaystyle\sum_{n=1}^M\int_{I_n}(\psi(w) - \psi(w_n),w - w_n)\right).$$
Furthermore,
$$
\begin{aligned}
J_2 &= \int_{\mathbb{R}^N}\int_{\mathbb{R}^N}\frac{1}{2}\frac{\left|\tau\displaystyle\sum_{n=1}^M(P_h\bar{w}^n(x) - P_h\bar{w}^n(z) - (W_n(x) - W_n(z)))\right|^2}{|x - z|^{N+2s}}\, dz\, dx\\
& \ge \int_{\mathbb{R}^N}\int_{\mathbb{R}^N}\frac{1}{4}\frac{\left|\tau\displaystyle\sum_{n=1}^M(\bar{w}^n(x) - \bar{w}^n(z) - (W_n(x) - W_n(z)))\right|^2}{|x - z|^{N+2s}}\, dz\, dx\\
& \quad - \int_{\mathbb{R}^N}\int_{\mathbb{R}^N}\frac{1}{2}\frac{\left|\tau\displaystyle\sum_{n=1}^M(\bar{w}^n(x) - \bar{w}^n(z) - (P_h\bar{w}^n(x) - P_h\bar{w}^n(z)))\right|^2}{|x - z|^{N+2s}}\, dz\, dx
\end{aligned}
$$
and
$$
\begin{aligned}
&\int_{\mathbb{R}^N}\int_{\mathbb{R}^N}\frac{1}{4}\frac{\left|\tau\displaystyle\sum_{n=1}^M(\bar{w}^n(x) - \bar{w}^n(z) - (W_n(x) - W_n(z)))\right|^2}{|x - z|^{N+2s}}\, dz\, dx \\
& = \int_{\mathbb{R}^N}\int_{\mathbb{R}^N}\frac{1}{4}\frac{|(\bar{w}(x) - \bar{w}(z)) - (\overline{W}(x) - \overline{W}(z))|^2}{|x - z|^{N+2s}}\, dz\, dx
\end{aligned}
$$
hold. Combine all the estimates together, we finally obtain
$$
\begin{aligned}
& \int_0^T(\psi(w) - \psi(W),w - W) \\
& \qquad + \int_{\mathbb{R}^N}\int_{\mathbb{R}^N}\frac{|(\bar{w}(x) - \bar{w}(z)) - (\overline{W}(x) - \overline{W}(z))|^2}{|x - z|^{N+2s}}\, dz\, dx\\
& \leq c\Bigg(\displaystyle\sum_{n=1}^M\int_{I_n}(\psi(w_n) - \psi(w),w_n - w) + \displaystyle\sum_{n=1}^M\int_{I_n}(\psi(w) - \psi(P_hw),w - P_hw)\\
& \qquad + \int_{\mathbb{R}^N}\int_{\mathbb{R}^N}\frac{\left|\tau\displaystyle\sum_{n=1}^M(\bar{w}^n(x) - \bar{w}^n(z) - (P_h\bar{w}^n(x) - P_h\bar{w}^n(z)))\right|^2}{|x - z|^{N+2s}}\, dz\, dx\\
& \qquad + \Vert\psi(w_0) - \Pi_h\psi(w_0)\Vert_{L^2(\mathbb{R}^N)}^2\Bigg).
\end{aligned}
$$
By \eqref{Eq5.1}, we get the desired estimate.
\end{proof}
The proof of Theorem \ref{thm:main5.1} also implies the following result.
\begin{corollary}
There is a constant $c>0$ independent of $h$ and $\tau$ such that
$$
\begin{aligned}
\int_0^T\Vert w - W\Vert_{(w,\frac{m+1}{m})}^2 & \leq c\Bigg(\displaystyle\sum_{n=1}^M\int_{I_n}\Vert w_n - w\Vert_{(w,\frac{m+1}{m})}^2 + \int_0^T\Vert w - P_hw\Vert_{(w,\frac{m+1}{m})}^2 \\
& \qquad + \Vert \psi(w_0) - \Pi_h\psi(w_0)\Vert_2^2\Bigg).
\end{aligned}
$$
\end{corollary}
Indeed, the corollary follows directly from the following estimate
\begin{equation*}
\begin{aligned}
& \int_{\mathbb{R}^N}\int_{\mathbb{R}^N}\frac{\left|\tau\displaystyle\sum_{n=1}^M(P_h\bar{w}^n(x) - P_h\bar{w}^n(z) - (W_n(x) - W_n(z)))\right|^2}{|x - z|^{N+2s}}\, dz\, dx\\
& \qquad + \int_0^T(\psi(w) - \psi(W),w - W) \\
& \leq c\Bigg(\displaystyle\sum_{n=1}^M\int_{I_n}(\psi(w_n) - \psi(w),w_n - w) + \displaystyle\sum_{n=1}^M\int_{I_n}(\psi(w) - \psi(P_hw),w - P_hw)\\
& \qquad + \Vert\psi(w_0) - \Pi_h\psi(w_0)\Vert_{L^2(\mathbb{R}^N)}^2\Bigg).
\end{aligned}
\end{equation*}

\section{The convergence rate.}\label{section6}
In the section we discuss the rates of convergence.
\begin{theorem}\label{thm:main6.1}
Let $0 < m < 1$ and $w_0 \in L^{\infty}(\mathbb{R}^N)\cap H^s_0(\Omega).$ Then there is a positive constant $c$ independent of $h$ and $\tau$ such that
\begin{equation}
\label{Eq7.1}
\begin{aligned}
\left(\int_0^T\Vert w - W\Vert_{(w,\frac{m+1}{m})}^2 + \displaystyle\sup_{1\leq n\leq N}\left\Vert\int_0^{t_n}(w - W)\right\Vert_{H^s(\mathbb{R}^N)}^2\right)^{1/2} \leq c(\tau + h^s).
\end{aligned}
\end{equation}
\end{theorem}
\begin{proof}
Set $v := |u|^{\frac{m-1}{2}}u$, by \cite[Lemma 4.1, Lemma 4.2]{Ebmeyer-Liu-2008}, we have
$$\int_{I_n}\Vert w_n - w\Vert_{(w,\frac{m+1}{m})}^2 \leq c\int_{I_n}(\psi(w_n) - \psi(w),w_n - w) \leq c\int_{I_n}\int_{\mathbb{R}^N}|v_n - v|^2.$$
Since
$$\int_{I_n}|v_n - v|^2 = \int_{I_n}\left |\int_t^{t_n}v_t\right|^2 \leq \int_{I_n}\int_t^{t_n}|v_t|^2|I_n| \leq \tau^2\int_{I_n}|v_t|^2$$
and $|v_t|^2 = |\partial_t(|u|^{\frac{m-1}{2}}u)|^2$, by the regularity results for nonlinear fractional diffusion equations from \cite{VPFR2017},
we obtain
\begin{equation}
\label{Eq7.2}
\displaystyle\sum_{n=1}^M\int_{I_n}\Vert w_n - w\Vert_{(w,\frac{m+1}{m})}^2 \leq c\tau^2\left\Vert \partial_tu^{\frac{m+1}{2}}\right\Vert_{L^2(0,T;L^2(\mathbb{R}^N))} \leq c\tau^2.
\end{equation}
The same arguments as Section 4.4 in \cite{Brenner-Scott} and the nonlocal Poincar$\acute{e}$ inequality $\mathbf{I}$ (see \cite[Lemma 4.3]{DuqiangSIAMReview}) imply that
\begin{equation}
\label{Eq7.3}
\begin{aligned}
& \int_0^T\Vert w - P_hw\Vert_{(w,\frac{m+1}{m})}^2 = \int_0^T\int_{\mathbb{R}^N}(|w| + |w - P_hw|)^{\frac{1-m}{m}}|w - P_hw|^2\\
& \leq c\Bigg(\Vert w\Vert_{L^{\infty(0,T;L^{\infty}(\mathbb{R}^N))}}^{\frac{1-m}{m}}\Vert w - P_hw\Vert_{L^2(0,T;L^2(\mathbb{R}^N))}^2 + \Vert w - P_hw\Vert_{L^{\frac{m+1}{m}}(0,T;L^{\frac{m+1}{m}}(\mathbb{R}^N))}^{\frac{m+1}{m}}\Bigg)\\
& \leq ch^{2s}\Bigg(\int_0^T\int_{\mathbb{R}^N}\int_{\mathbb{R}^N}\frac{(w(x) - w(z))^2}{|x - z|^{N+2s}}\, dz\, dx\, dt + \Vert w \Vert_{L^{\frac{m+1}{m}}(0,T;W^{\frac{2ms}{m+1},\frac{m+1}{m}}(\mathbb{R}^N))}^{\frac{m+1}{m}}\Bigg)\\
& \leq ch^{2s}.
\end{aligned}
\end{equation}
From the regularity results of \cite{VPFR2017}, we have
\begin{equation}
\label{Eq7.4}
\begin{aligned}
& \int_{\mathbb{R}^N}\int_{\mathbb{R}^N}\frac{\left|\tau\displaystyle\sum_{n=1}^M[(\bar{w}^n(x) - \bar{w}^n(z)) - (P_h\bar{w}^n(x) - P_h\bar{w}^n(z))]\right|^2}{|x - z|^{N+2s}}\, dz\, dx\\
& = \int_{\mathbb{R}^N}\int_{\mathbb{R}^N}\frac{\left|(\bar{w}(x) - \bar{w}(z)) - (P_h\bar{w}(x) - P_h\bar{w}(z))\right|^2}{|x - z|^{N+2s}}\, dz\, dx\\
& \leq ch^{2s}\int_{\mathbb{R}^N}\int_{\mathbb{R}^N}\frac{\left|\bar{w}(x) - \bar{w}(z)\right|^2}{|x - z|^{N+2s}}\, dz\, dx\\
& \leq ch^{2s}.
\end{aligned}
\end{equation}
Furthermore, since $w_0 \in L^{\infty}(\mathbb{R}^N)\cap H^s_0(\Omega)$, $\psi(w_0) \in H^s_0(\Omega)$ and
\begin{equation}
\label{Eq7.5}
\begin{aligned}
& \Vert \psi(w_0) - \Pi_h\psi(w_0)\Vert_{L^2(\mathbb{R}^N)}^2 \leq ch^{2s}\int_{\mathbb{R}^N}\int_{\mathbb{R}^N}\frac{(\psi(w_0)(x) - \psi(w_0)(z))^2}{|x - z|^{N+2s}}\, dz\, dx \leq ch^{2s}.
\end{aligned}
\end{equation}
From Theorem \ref{thm:main5.1} and \eqref{Eq7.2}-\eqref{Eq7.5}, there holds
\begin{equation}
\label{Eq7.6}
\begin{aligned}
& \int_{\mathbb{R}^N}\int_{\mathbb{R}^N}\frac{|(\bar{w}(x) - \bar{w}(z)) - (\overline{W}(x) - \overline{W}(z))|^2}{|x - z|^{N+2s}}\, dz\, dx + \int_0^T\Vert w - W\Vert_{(w,\frac{m+1}{m})}^2\\
& \leq c(\tau^2 + h^{2s}),
\end{aligned}
\end{equation}
for $\bar{w}(x) = \int_0^Tw(x,t)\, dt$, $\overline{W}(x) = \int_0^TW(x,t)\, dt$ and $T = t_M$. By the nonlocal Poincar$\acute{e}$ inequality $\mathbf{I}$ (see \cite[Lemma 4.3]{DuqiangSIAMReview}) and \cite[Lemma 4.2]{DuqiangSIAMReview} again, there exists a positive constant $c$ such that
$$
\begin{aligned}
& \left\Vert\int_0^{t_n}(w - W)\right\Vert_{L^2(\mathbb{R}^N)}^2 \leq c\int_{\mathbb{R}^N}\int_{\mathbb{R}^N}\frac{\left(\int_0^{t_n}(w - W)(x) - \int_0^{t_n}(w - W)(z)\right)^2}{|x - z|^{N+2s}}\, dz\, dx
\end{aligned}
$$
and the assertion follows.
\end{proof}
\begin{corollary}
Under the assumptions of Theorem \ref{thm:main6.1}, there exists a positive constant $c$ independent of $h$ and $\tau$ such that
$$\Vert u - U\Vert_{L^{m+1}(0,T;L^{m+1}(\mathbb{R}^N))} \leq c(\tau + h^s),$$
where $U = \psi(W)$.
\end{corollary}
The proof of the above corollary is similar to that of \cite[Corollary 6.2]{Ebmeyer-Liu-2008}, hence we omit it.
\begin{corollary}
Under the assumptions of Theorem \ref{thm:main6.1}, then
$$
\begin{aligned}
\left(\int_0^T\int_{\mathbb{R}^N}|u - U||w - W| \right)^{\frac{1}{2}} + \left\Vert\int_0^t(w - W)\right\Vert_{L^{\infty}(0,T;H^s(\mathbb{R}^N))} \leq c(\tau + h^s).
\end{aligned}
$$
\end{corollary}
\begin{proof}
Assume $T_0 \in (t_{n-1},t_n)$, note that $W$ is constant with respect to $t$ in $(t_{n-1},t_n]$, the $M$th equation of \eqref{Eq4.2} can be written as
$$
\begin{aligned}
& \left(\frac{\psi(W(T_0)) - \psi(W(t_{n-1}))}{\tau},\chi_M \right)\\
&+ \int_{\mathbb{R}^N}\int_{\mathbb{R}^N}\frac{W(T_0)(x) - W(T_0)(z)}{|x - z|^{N+2s}}\big(\chi_M(x) - \chi_M(z)\big)\, dz\, dx = 0 \quad \forall\chi_M \in S_h
\end{aligned}
$$
Hence, replacing $t_M$ by $T_0$ and $W_M$ by $W(T_0)$ in the proofs of Theorem \ref{thm:main5.1} and Theorem \ref{thm:main6.1}, we obtain
$$
\begin{aligned}
\left\Vert\int_0^{T_0}(w - W)\right\Vert_{H^s(\mathbb{R}^N)} \leq c(\tau + h^s).
\end{aligned}
$$
\end{proof}

\section{Some extensions}\label{section7}
In this section, we indicate how to generalize the results in Section 3 and Section 4 to a more general parabolic integral equation. Let $\Omega_s\subset\mathbb{R}^n$ and $\Omega_c\subset\mathbb{R}^n$ be bounded and open polyhedral domains and $T\in (0, +\infty)$, $\Omega_s$ and $\Omega_c$ have a nonempty common boundary, we consider the following nonlinear diffusion problem
\begin{equation}\label{Eq8.1}
\begin{cases}
u_t = \mathcal{L}(|u|^{m-1}u) &\quad \mathrm{on} \,\, \Omega_s , t > 0,\\
\mathcal{V}u = 0  &\quad \mathrm{on} \,\, \Omega_c, t > 0,\\
u(x,0) = u_0(x) &\quad \mathrm{on} \,\, \Omega_s\cup \Omega_c,
\end{cases}
\end{equation}
where $u(x,t):(\Omega_s\cup\Omega_c) \times [0,T] \to \mathbb{R}$ and $\mathcal{V}$ denotes a linear operator of constraints acting on a volume $\Omega_c$ which is disjoint from $\Omega_s$.

Given an open bounded subset $\Omega_s \subset \mathbb{R}^n$, $\Omega_c$ is the corresponding {\it interaction domain}. For $u(x) : \Omega \to \mathbb{R}$, the action of the linear operator $\mathcal{L}$ on the function $u(x)$ is defined as
\begin{equation}
\label{Eq8.2}
\mathcal{L}u(x):= 2\int_{\Omega_s\cup \Omega_c}\big(u(y) - u(x)\big)\gamma (x,y)\, dy\qquad \forall x\in \Omega_s \subseteq \mathbb{R}^n,
\end{equation}
where the volume of $\Omega_s$ is nonzero and the kernel $\gamma(x,y)$ is a nonnegative symmetric mapping, i.e., $\gamma (x,y) = \gamma (y,x) \ge 0$. We refer the interested readers to \cite{DuqiangSIAMReview} for more details.

As in \cite{DuqiangSIAMReview}, given positive constants $\gamma_0$ and $\varepsilon$, assume that $\gamma$ satisfies
\begin{equation}
\label{Eq8.3}
\begin{cases}
\gamma(x,y) \ge 0 & \quad \forall y \in B_{\varepsilon}(x)\,\, \mathrm{and}\,\, \gamma(x,y) \ge \gamma_0 > 0 \quad \mathrm{when}\,\, y \in B_{\varepsilon/2}(x),\\ \gamma(x,y) = 0 & \quad \forall y \in (\Omega_s\cup\Omega_c)\setminus B_{\varepsilon}(x),
\end{cases}
\end{equation}
where $B_{\varepsilon}(x) := \{y \in \Omega_s\cup\Omega_c : |y - x| \leq \varepsilon\}$, for all $x \in \Omega_s\cup\Omega_c$. Furthermore, assume that
there exist $s \in (0,1)$ and positive constants $\gamma_*$ and $\gamma^*$ such that, for all $x \in \Omega_s$,
\begin{equation}\label{case1}
\frac{\gamma_*}{|y - x|^{n+2s}} \leq \gamma(x,y) \leq \frac{\gamma^*}{|y - x|^{n+2s}} \qquad \text{for}\,\, y \in B_{\varepsilon}(x).
\end{equation}

For $s \in (0,1)$, define the fractional-order Sobolev space as
$$H^s(\Omega_s\cup\Omega_c) := \{u \in L^2(\Omega_s\cup\Omega_c) : \Vert u\Vert_{L^2(\Omega_s\cup\Omega_c)} + |u|_{H^s(\Omega_s\cup\Omega_c)} < \infty\},$$
where $|u|_{H^s(\Omega_s\cup\Omega_c)}^2 := \int_{\Omega_s\cup\Omega_c}\int_{\Omega_s\cup\Omega_c}(u(y) - u(x))^2\gamma(x,y)\, dy\, dx$.
Assuming $u_0 \in L^{\infty}(\Omega_s\cup\Omega_c)$, $u(x,t)$ a weak solution of \eqref{Eq8.1} if
\begin{equation}
\label{Eq8.5}
\begin{aligned}
& - \int_0^T\int_{\Omega_s\cup \Omega_c}u\varphi_t\, dx\, dt\\
& \quad + \int_0^T\int_{\Omega_s\cup \Omega_c}\int_{\Omega_s\cup \Omega_c}\big((|u|^{m-1}u)(y) - (|u|^{m-1}u)(x)\big)\gamma(x,y)\big(\varphi(y) - \varphi(x)\big)\, dy\, dx\, dt\\
& = \int_{\Omega_s\cup \Omega_c} u_0\varphi_0\, dx,
\end{aligned}
\end{equation}
where $\varphi(\cdot,T) \equiv 0$ on $\Omega_s$, $\varphi_0 = \varphi(\cdot,0)$.
Let $T_h$ be a family of decompositions of $\Omega_s$ into closed $N-$simplices and $h$ is the mesh-size, assume $T_h$ is a regular triangulation and there exists a constant $c > 0$ such that
$$|K| \ge c(\text{diam} \,\,\, K)^N \text{ for all simplices }K \in T_h.$$
Introduce the space
\begin{center}
$S_h(\Omega_s\cup\Omega_c) =$ \{$\phi_h \in C^0(\overline{\Omega}):\phi_h$ is piecewise linear w.r.t. $T_h$, $\mathcal{V}\phi_h = 0$ on $\Omega_c$\}.
\end{center}
Let $\Pi_hv \in S_h$ denote the $C^0-$piecewise linear interpolant of the function $v$ and $P_h : H^s(\Omega_s\cup\Omega_c) \to S_h$ be the $H^s-$projection onto $S_h$ defined by
$$ \int_{\Omega_s\cup \Omega_c}\int_{\Omega_s\cup \Omega_c}\big((v - P_hv)(y) - (v - P_hv)(x)\big)\gamma(x,y)\big(\chi(y) - \chi(x)\big)\, dy\, dx = 0 \quad \forall\chi \in S_h.$$
As in Section 2, set $w := |u|^{m-1}u$, then $\partial_tu = \mathcal{L}w$ and
\begin{equation}
\label{Eq8.6}
\partial_t\psi(w) = \mathcal{L}w,
\end{equation}
where $\psi(s) := |s|^{\frac{1-m}{m}}s$. Let $W_n \in S_h, n = 1,2,\ldots,$ be the solutions of the system
\begin{equation}
\label{Eq8.7}
\begin{aligned}
& \Big(\frac{\psi(W_n) - \psi(W_{n-1})}{\tau},\chi_n \Big)\\
&+ \int_{\Omega_s\cup \Omega_c}\int_{\Omega_s\cup \Omega_c}\big(W_n(y) - W_n(x)\big)\gamma(x,y)\big(\chi_n(y) - \chi_n(x)\big)\, dy\, dx = 0 \quad \forall\chi_n \in S_h,
\end{aligned}
\end{equation}
where $W_0 := \psi^{-1}(\Pi_h\psi(w_0))$. Then the finite element approximation $W(x,t)$ of $w(x,t)$ is defined as
\begin{equation}
\label{Eq8.8}
W(x,t) =
\begin{cases}
W_0(x) &\quad \mathrm{if} \,\, t = 0,\\
W_n(x,t)&\quad \mathrm{if} \,\, t \in (t_{n-1},t_n], 1 \leq n \leq N.
\end{cases}
\end{equation}
The same arguments of Section 3 and Section 4 imply the following results.
\begin{theorem}\label{thm:main7.6.1}
For any $m > 0$ there is a positive constant $c$ independent of $h$ and $\tau$ such that
$$
\begin{aligned}
& \int_0^T\Vert w - W\Vert_{(w,\frac{m+1}{m})}^2 + \int_{\Omega_s\cup\Omega_c}\int_{\Omega_s\cup\Omega_c}|(\bar{w}(y) - \bar{w}(x)) - (\overline{W}(y) - \overline{W}(x))|^2\gamma(x,y)\, dy\, dx\\
& \leq c\Big(\displaystyle\sum_{n=1}^N\int_{I_n}\Vert w_n - w\Vert_{(w,\frac{m+1}{m})}^2 + \int_0^T\Vert w - P_hw\Vert_{(w,\frac{m+1}{m})}^2 + \Vert \psi(w_0) - \Pi_h\psi(w_0)\Vert_2^2\\
&\quad + \int_{\Omega_s\cup\Omega_c}\int_{\Omega_s\cup\Omega_c}\left|\tau\displaystyle\sum_{n=1}^N[(\bar{w}^n(y) - \bar{w}^n(x)) - (P_h\bar{w}^n(y) - P_h\bar{w}^n(x))]\right|^2\gamma(x,y)\, dy\, dx  \Big),
\end{aligned}
$$
where $\bar{w}(x) = \int_0^Tw(x,t)\, dt$ and $\overline{W}(x) = \int_0^TW(x,t)\, dt$.
\end{theorem}
\begin{theorem}\label{thm:main7.6.2}
Let $0 < m < 1$, $w_0 \in L^{\infty}(\Omega_s\cup\Omega_c)\cap H^s(\Omega_s\cup\Omega_c)$ and $\mathcal{V}w_0 = 0$ on $\Omega_c$. Then there is a positive constant $c$ independent of $h$ and $\tau$ such that
\begin{equation}
\label{Eq8.15}
\begin{aligned}
\left(\int_0^T\Vert w - W\Vert_{(w,\frac{m+1}{m})}^2 + \displaystyle\sup_{1\leq n\leq N}\left\Vert\int_0^{t_n}(w - W)\right\Vert_{H^s(\Omega_s\cup\Omega_c)}^2\right)^{1/2} \leq c(\tau + h^s).
\end{aligned}
\end{equation}
\end{theorem}
\begin{corollary}\label{thm:main7.6.3}
Under the assumptions in the Theorem \ref{thm:main7.6.1}, there exists a positive constant $c$ independent of $h$ and $\tau$ such that
$$\Vert u - U\Vert_{L^{m+1}(0,T;L^{m+1}(\Omega_s\cup\Omega_c))} \leq c(\tau + h^s),$$
where $U = \psi(W)$.
\end{corollary}
\begin{corollary}\label{thm:main7.6.4}
Under the assumptions in the Theorem \ref{thm:main7.6.1}, then
$$
\begin{aligned}
\left(\int_0^T\int_{\Omega_s\cup\Omega_c}|u - U||w - W| \right)^{\frac{1}{2}} + \left\Vert\int_0^t(w - W)\right\Vert_{L^{\infty}(0,T;H^s(\Omega_s\cup\Omega_c))} \leq c(\tau + h^s).
\end{aligned}
$$
\end{corollary}

\end{document}